\documentclass[10pt,a4paper,twoside]{article}
\usepackage{amsthm,amsfonts,amssymb,amsmath,amscd,bezier,color}
\usepackage{graphicx}
\usepackage{hyperref}
\usepackage[all]{xy}

\newcommand{\Z}{\ensuremath{\mathbb{Z}}}
\newcommand{\R}{\ensuremath{\mathbb{R}}}

\newcommand{\C}{\ensuremath{\mathbb{C}}}
\newcommand{\RP}{\ensuremath{\mathbb{R}\mathrm{P}}}
\newcommand{\p}{\ensuremath{\mathfrak{p}}}
\newcommand{\q}{\ensuremath{\mathfrak{q}}}
\newcommand{\A}{\ensuremath{\mathfrak{a}}}

\newcommand{\im}{\ensuremath{\mathrm{im}}}

\newtheorem{definition}{Definition}[section]
\newtheorem{proposition}[definition]{Proposition}
\newtheorem{lemma}[definition]{Lemma}

\newtheorem{theorem}[definition]{Theorem}
\newtheorem{corollary}[definition]{Corollary}

\begin{document}

\begin{center}

{\Large Roots of maps between spheres and projective \\ \vspace{2mm} spaces in codimension one} 

\vspace{10mm}

{\large M.\,C.\,Fenille, D.\,L.\,Gon\c calves and G.\,L.\,Prado}

\vspace{5mm}
\end{center}


{\bf Abstract:} For maps from $S^3$ and $\RP^3$ into $S^2$ and $\RP^2$, we study the problem of minimizing the root set by deforming the maps through homotopies. After presenting the classification of the homotopy classes of such maps, we prove that the minimal root set for a non null-homotopic map is either a circle or the disjoint union of two circle, according its range is $S^2$ or $\RP^2$, respectively.

\vspace{5mm}

\vspace{2mm}


\section{Introduction}\label{Section-Introduction}

\hspace{5mm} This article approaches, in specific settings, one of the major problems in topological root theory, namely, the question of minimizing the root set of a map by deforming it through homotopies. 

We consider maps from the $3$-sphere $S^3$ and from the $3$-dimensional projective space $\RP^3$ into the $2$-sphere $S^2$ and the projective plane $\RP^2$. Thus, we have the case of positive codimension and so a context in which the minimizing problem has been very little  explored, which includes its formulation.

One way to handle the problem is in terms of number of components and the cohomological dimension of the root or coincidence set; see \cite{Daciberg-Handbook}.

We mention some works that approach the problem in positive codimension in the most general setting of coincidences.

Article \cite{K1} concerns pairs of maps $f,g:M\to N$ between manifolds of arbitrary dimensions and the minimum number ${\rm MCC}(f,g)$ of path components of the coincidence set ${\rm Coin}(f',g')$ for maps $f'$ and $g'$ homotopic to $f$ and $g$, respectively. The problem is approached via normal bordism theory and a lower bound $N(f,g)$ for ${\rm MCC}(f,g)$ is presented, generalizing the Nielsen coincidence number (well known for maps between manifolds of the same dimension). 

In \cite{K2}, a single map $f:M\to N$ between compact smooth manifolds is considered and a normal bordism class $\omega(f)$ is defined in such a way that the possibility of annihilating the self-coincidences of $f$ is given in terms of the nullity of $\omega(f)$.

As the main goal of \cite{K3} the author computes for all pairs of maps $f_1,f_2:S^m\to\mathbb{K}\mathrm{P}(n)$, from a sphere into a real, complex or quaternionic projective space, the following invariants: (i) the strong coincidence Nielsen number, (ii) the minimal number of components ${\rm MCC}(f_1,f_2)$, and (iii) the minimal number ${\rm MC}(f_1,f_2)$ of coincidence points among all maps in the homotopy classes of the given pair.

In \cite{KP}, pairs of maps $f,g:S^{n+1}\to Y$ from the $(n+1)$-sphere into a connected smooth $n$-manifold are considered and a Wecken type theorem is proved.

We highlight that the cited articles \cite{K3} and \cite{KP} concern, among other, maps between spheres and from spheres into projective spaces, which overlap with some of the cases we consider in this note. When we see an interesting relationship between our results and those of \cite{K3} and \cite{KP}, we make specific citations and suggest comparisons.

In a context in which the involved spaces are not necessarily manifolds, the main theorem of \cite{Fenille} presents a necessary and sufficient condition for a pair of maps from a finite and connected $2$-dimensional $CW$ complex into a graph can be deformed to be coincidence free. The condition is given in terms of the existence of a lifting in a certain diagram of fundamental groups and induced homomorphisms. 

Specifically on the root problem, we highlight the article \cite{Go2} in which for a map $f:N_1^m\to N_2^n$ between compact nilmanifolds of the indicated dimensions, it is proved that the Nielsen root number $N(f)$ is non-zero if and only if the Hirsch lenght of $\ker(f_{\#})$ is $m-n$, and in this case, $N(f)$ is finite and equal to the Reidemeinster number $R(f)=[\pi_1(N_2) : f_{\#}\pi_1(N_1)]$, and the $\check{\rm C}$ech cohomology group $\check{H}^{m-n}(F;\Z)$ is not trivial for each essential Nielsen root class $F$ of $f$.

The quite delicate problem of finding the minimal number of roots for maps between surfaces has been settled for the orientable cases in the articles \cite{BGZ} and \cite{GZ}. This is a case of codimension zero but in low dimension which uses a quite different approach from the case of higher dimension.

Here we consider just a few low dimension cases, perhaps the simplest one,  but under the hypothesis of positive codimension, in fact codimension one. The questions to be addressed are of the same natura, but the study uses quite a different approach.

For studying maps from $S^3$ into $S^2$, we take advantage on the known fact that the homotopy set $[S^3,S^2]$ is in fact the third homotopy group $\pi_3(S^2)$, which is the infinite cyclic group  generated by the Hopf fibration $h:S^3\to S^2$, in such a way that a representative map $nh:S^3\to S^2$ for the homotopy class $n[h]$, corresponding to $n$ under the isomorphism $\pi_3(S^2)\approx\Z$ mapping $[h]$ to $1$, is given by the composition of $h$ with a self-map of $S^3$ of degree $n$; see \cite{Whitehead}.

For maps from $\RP^3$ into $S^2$, the first problem is to establish the homotopy classification and describe representative maps for each homotopy class in $[\RP^3,S^2]$. We solve this problem in Sections \ref{Section-Hopf-Map} and \ref{Section-Homotopy-Classes}. First, in Proposition \ref{Proposition-Factorization-Hopf-Map}, for each integer $n$, we define a map $h'_n:\RP^3\to S^2$ such that $h'_n\circ\p_3$ is homotopic to $nh$, where $\p_3:S^3\to\RP^3$ is the double covering map. Then, in Theorem \ref{Theorem-Homotopy-Class}, we prove that the function $[h'_n]\mapsto[h_n]$ is a bijection from the set $[\RP^3,S^2]$ onto the group $\pi_3(S^2)$. The main results of Sections \ref{Section-Hopf-Map} and \ref{Section-Homotopy-Classes} are in the PhD Thesis \cite{Prado-Tese}.

In Section \ref{Section-Roots} we solve the root problem for maps into $S^2$ as follows: Let $f$ be a map from either $S^3$ or $\RP^3$ into $S^2$ and let $y_0\in S^2$ be a base point, respect to which we consider roots. If $f$ is null-homotopic, then it is homotopic to a constant map at a point $y_1\neq y_0$ and so $f$ is homotopic to a root free map. Otherwise, $f$ is {\it strongly surjective} (i.e. each map homotopic to $f$ is surjective) and we prove (Theorems \ref{Theorem-Hopf-Minimum} and \ref{Theorem-HopfLine-Minimum}) that $f$ is homotopic to a map $\varphi$ such that $\varphi^{-1}(y_0)$ is a circle. Moreover, Theorem \ref{Theorem-H1neq0} states that for each map $g$ homotopic to $f$, one has ${\rm rank}(\check{H}^1(g^{-1}(y_0)))\geq1$, which means that one can not expect to have the root set $g^{-1}(y_0)$ properly contained in a circle. Therefore, it is reasonable to state that, in this situation, a minimal set for the root problem is a circle. 


These results suggest the following formulation of a {\it minimal root set} of a homotopy class, in a general context.

\begin{definition}\label{Definition-Minimal}
{\rm Given a point $y_0\in Y$ and a homotopy class $\alpha\in[X,Y]$, we say that the root set $f_0^{-1}(y_0)$, for a map $f_0\in\alpha$, is {\it minimal}, provided there is no $f\in\alpha$ such that $f^{-1}(y_0)$ is a proper subset of $f_0^{-1}(y_0)$.}
\end{definition}

Of course, Definition \ref{Definition-Minimal} may be easily reformulated for introduce the concepts of {\it minimal fixed point set} and {\it minimal coincidence set}. 



In Section \ref{Section-Maps-Into-RP2} we solve the root problem for maps from both $S^3$ and $\RP^3$ into $\RP^2$. We start (Lemma \ref{Lemma-Lifting}) by describing bijections from the sets $[S^3,S^2]$ and $[\RP^3,S^2]$ onto the sets $[S^3,\RP^2]$ and $[\RP^3,\RP^2]$, respectively. Then, we prove analogous results to those of Section \ref{Section-Roots}. Precisely, we prove (Theorem \ref{Theorem-Min-Roots-RP2}) that the minimal root set for a non null-homotopic map from either $S^3$ or $\RP^3$ into $\RP^2$ is the disjoint union of two circle. 


Throughout the text, we simplify $f$ is a continuous map by $f$ is a map, $[f]$ is the homotopy class of $f$, and the symbols $\simeq$ and $\approx$ indicate homotopy of maps and isomorphism of groups and rings, respectively. Additionally, $\p_n:S^n\to\RP^n$ is the double covering map and $\q_n:\RP^n\to S^n$ is the map which collapses the $(n-1)$-skeleton $\RP^{n-1}\subset\RP^n$ to a point.


\section{Factorization of the Hopf fibration}\label{Section-Hopf-Map}

\hspace{4mm} \hspace{4mm} Let $h:S^3\to S^2$ be the Hopf fibration. We briefly describe $h$ in order to establish the notation. Consider the sphere $S^3$ as the subset of $\C^2$ given by $$S^3=\{(z_1,z_2)\in\C^2 : |z_1|^2+|z_2|^2=1\}.$$

Consider the action $S^1\times S^3\to S^3$ of the circle over $S^3$ defined by $(\lambda,(z_1,z_2))\mapsto(\lambda z_1,\lambda z_2 )$. The orbit of each point $(z_1,z_2)\in S^3$ under this action is a circle which we denote by $S^1(z_1,z_2)$. The quotient space $S^3/S^1$ of $S^3$ by this action is the complex projective space $\C{\rm P}^1\cong S^2$. The Hopf map $h:S^3\to S^2$ is defined by mapping each point $(z_1,z_2)\in S^3$ in the class $[S^1(z_1,z_2)]\in S^3/S^1\cong S^2$ of its orbit. 

It is well known that the third homotopy group $\pi_3(S^2)\approx\Z$ under an isomorphism such that the homotopy class $[h]$ of $h$ corresponds to $1\in \mathbb{Z}$.

For each $n\in\Z$, we consider maps $nh:S^3\to S^2$ which belong to the homotopy class $n[h]$. There exist several ways to construct such maps; a map $nh$ is obtained as the composite of $h$ with any self-map of $S^3$ of degree $n$.

Henceforward, $\p_3:S^3\to\RP^3$ is the double covering map obtained by collapsing antipodal points, which
we write as $$S^3\ni(z_1,z_2) \ \mapsto \  [z_1,z_2]\in \RP^3.$$

Of course, for each map $f:\RP^3\to S^2$, the composed map $f\circ\p_3:S^3\to S^2$ is homotopic to $nh$ for some integer $n$. Proposition \ref{Proposition-Factorization-Hopf-Map} below shows that, for each integer $n$, we can get a map homotopic to $nh$ via such a composition.

We highlight that the notations and constructions presented in the proof of Proposition \ref{Proposition-Factorization-Hopf-Map} are so important as the result itself.

\begin{proposition}\label{Proposition-Factorization-Hopf-Map}
For each integer $n$, there exists a map $h'_n:\RP^3\to S^2$ such that $$nh\simeq h'_n\circ\p_3.$$
\end{proposition}
\begin{proof}
Of course, the map $h':\RP^3\to S^2$ given by $h'([z_1,z_2])=[S^1(z_1,z_2)]$ is well defined, continuous and satisfies $h'\circ\p_3=h$. This proves the result for $n=1$. For $n=0$, we take $h'_0$ to be the constant map at $[S^1(1,0)]$. For the remaining proof, we consider separately the cases  $n$ odd and $n$ even.

We consider first the odd indexes. Given an odd integer $k$, let us consider the maps $\A_k:S^3\to S^3$ and $\A_k':\RP^3\to\RP^3$ given by $$\A_k(z_1,z_2)=(z_1^k,z_2)/|(z_1^k,z_2)| \quad{\rm and}\quad \A'_k([z_1,z_2])=[\A_k(z_1,z_2)].$$

The map $\A_k$ has degree $k$ and so $h\circ\A_k\simeq kh$. Moreover, $\p_3\circ \A_k=\A'_k\circ\p_3$.

Define $h'_k:\RP^3\to S^2$ by setting $h'_k=h'\circ \A'_k$. Then $$h'_k\circ\p_3=h'\circ\A'_k\circ \p_3=h'\circ\p_3\circ\A_k=h\circ\A_k\simeq kh.$$

The proof of the odd case is complete.

For the even indexes, we start by considering the quotient map $\q_3:\RP^3\to S^3$ defined by collapsing the $2$-skeleton $\RP^2\subset\RP^3$ to a point $s_0$. Then $\deg(\q_3)=1$ (for an appropriated choice of the orientations) and so the homotopy class $[q_3]\in[\RP^3,S^3]$ corresponds, under the bijection given by the Hopf-Whitney Classification Theorem \cite[Corollary 6.19, p.\,244]{Whitehead}, to a generator of $H^3(\RP^3;\Z)\approx\Z$. It follows that the map $\q_3\circ\p_3:S^3\to S^3$ has degree $2$ and so $h\circ\q_3\circ\p_3\simeq2h$. Thus, the map $h'_2=h\circ\q_3:\RP^3\to S^2$ satisfies the condition $h'_2\circ\p_3\simeq 2h$.

Finally, for each $n\in\Z$, define $h'_{2n}:\RP^3\to S^2$ by setting $h'_{2n}=nh\circ\q_3$. Then, $$h'_{2n}\circ\p_3=nh\circ\q_3\circ\p_3\simeq 2nh.$$

The proposition is proved.
\end{proof}

The constructions made in the proof of Proposition \ref{Proposition-Factorization-Hopf-Map} are illustrated in the following diagrams: the one on the left is for the odd indexes and the one on the right is for the even indexes. $$
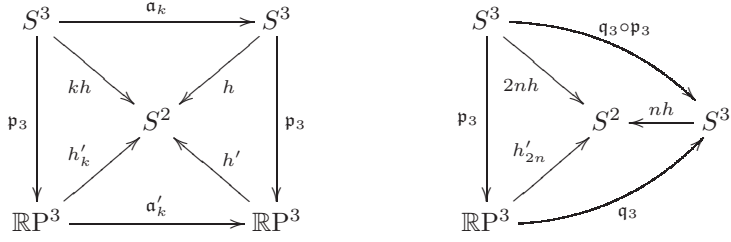
\xymatrix{
    S^3 \ar[rr]^{\A_k} \ar[dr]_{kh} \ar[dd]_{\p_3} & & S^3 \ar[dd]^{\p_3} \ar[dl]^{h} & & S^3 \ar[dd]_-{\p_3} \ar[dr]_-{2nh} \ar@/^0.4cm/[drr]^-{\q_3\circ\p_3} & & 
    \\  & S^2 & & & & S^2 & S^3 \ar[l]_-{nh} 
    \\ \RP^3 \ar[rr]^{\A'_k} \ar[ur]^{h'_k} & & \RP^3 \ar[ul]_{h'} & & \RP^3 \ar[ur]^-{h'_{2n}} \ar@/_0.4cm/[urr]_-{\q_3} & & }$$

In the next sections we use the constructions and notations established here.


\section{Homotopy classes of maps for $\RP^3$ into $S^2$}\label{Section-Homotopy-Classes}

\hspace{4mm} As before, let $\p_3:S^3\to\RP^3$ be the double covering map. Consider the function $$\hat{\p}_3:[\RP^3,S^2]\to\pi_3(S^2) \quad{\rm given \ by}\quad \hat{\p}_3([f])=[f\circ\p_3].$$

The following theorem is the main result of this section.

\begin{theorem}\label{Theorem-Homotopy-Class}
$[\RP^3,S^2]=\{[h'_n] : n\in\Z\}$ and the function $\hat{\p}_3$ is a bijection.
\end{theorem}

We remark that, once proved the equality $[\RP^3,S^2]=\{[h'_n] : n\in\Z\}$, the claim that $\hat{\p}_3$ is a bijection follows from Proposition \ref{Proposition-Factorization-Hopf-Map}, which states that $\hat{\p}_3([h_n'])=nh$.

Before proving the equality in Theorem \ref{Theorem-Homotopy-Class}, we present two lemmas.

We start by considering the double covering map $\p_2:S^2\to{\RP}^2$ and the quotient map $\q_2:\RP^2\to S^2$ which collapses the $1$-skeleton $S^1\subset\RP^2$ to a point.

Consider the Barratt-Puppe sequence associated to the map $\p_2$, namely,
\begin{equation}\label{Equation-Puppe-Sequence}
    S^2\stackrel{\p_2}{\longrightarrow}\RP^2\stackrel{\ell_2}{\longrightarrow}\RP^3\stackrel{\q_3}{\longrightarrow}S^3\stackrel{\p_2^1}{\longrightarrow}\Sigma\RP^2\stackrel{\ell_2^1}{\longrightarrow}\Sigma\RP^3\to
\end{equation}

Here, the map $\ell_2$ is the natural inclusion and we are considering the natural homeomorphisms $\RP^3\!/\RP^2\cong S^3$ and $\Sigma S^2\cong S^3$. Further, $\p_2^1$ and $\ell_2^1$ denote the suspensions of $\p_2$ and $\ell_2$, respectively. For details about properties of the Barratt-Puppe sequence, see \cite[Chapter 2]{SW}.

Consider the corresponding exact sequence of sets of homotopy classes, namely,
\begin{equation}\label{Equation-Exact-Sequence-From-Puppe}
\to\left[\Sigma\RP^2,S^2\right]\stackrel{\hat{\p}_2^1}{\longrightarrow}[S^3,S^2]\stackrel{\hat{\q}_3}{\longrightarrow}[\RP^3,S^2]\stackrel{\hat{\ell}_2}{\longrightarrow}[\RP^2,S^2]\stackrel{\hat{\p}_2}{\longrightarrow}[S^2,S^2].
\end{equation}

We observe that, in this sequence, just the second and third sets from the right do not have a natural structure of group. However, there exists a natural action of $\pi_3(S^2)$ on $[\RP^3,S^2]$ that we use in the proof of Theorem \ref{Theorem-Homotopy-Class}.

For $k=2,3$ we have $[S^k,S^2]=\pi_k(S^2)\approx\Z$. In the next two lemmas we describe the set $[\RP^2,S^2]$, the group $[\Sigma\RP^2,S^2]$ and the functions $\hat{\p}_2$ and $\hat{\p}_2^{1}$. In fact, we prove that both sets have cardinality two (so the second one is isomorphic to $\Z_2$) and both functions are {\it constant}, more precisely, the image of each of them is just the corresponding homotopy class of a constant map.

\begin{lemma}\label{Lemma-p2}
The set $[\RP^2,S^2]=\{0,[\q_2]\}$ and $\hat{\p}_2$ is the null function.
\end{lemma}
\begin{proof}
By the Hopf-Whitney Classification Theorem, one has an isomophism $[S^2,S^2]\approx H^2(S^2;\Z)\approx\Z$ (corresponding to the {\it degree function}) and a bijection of sets $$[{\RP}^2,S^2]\equiv H^2({\RP}^2;\Z)\approx\Z_2.$$ 

It is easy to see that $\q_2^{\ast}:H^2(S^2;\Z)\to H^2(\RP^2;\Z)$ corresponds to the quotient homomorphism $\Z\to\Z_2$, which implies that $\q_2$ is not null-homotopic. It follows that $[{\RP}^2,S^2]=\{0,[\q_2]\}$. Since $\p_2^{\ast}\circ\q_2^{\ast}:H^2(S^2;\Z)\to H^2(S^2;\Z)$ is the trivial homomorphism, we have $\hat{\p}_2([\q_2])=[\q_2\circ\p_2]=0$. Therefore, $\hat{\p}_2$ is null.
\end{proof}

\begin{lemma}\label{Lemma-p21}
The group $[\Sigma\RP^2,S^2]\approx\Z_2$ and the homomorphism $\hat{\p}_2^{1}$ is null.
\end{lemma}
\begin{proof}
By considering the obvious homeomorphisms $\RP^2/S^1\cong S^2$ and $\Sigma^kS^1\cong S^{k+1}$, the Barratt-Puppe sequence associated to the double covering map $\p_1:S^1\to S^1$ is $$S^1\stackrel{\p_1}{\longrightarrow}S^1\stackrel{\ell_1}{\longrightarrow}\RP^2\stackrel{\q_2}{\longrightarrow}S^2\stackrel{\p_1^1}{\longrightarrow} S^2\stackrel{\ell_1^1}{\longrightarrow}\Sigma\RP^2\stackrel{\q_2^1}{\longrightarrow} S^3\stackrel{\p_1^2}{\longrightarrow} S^3\to$$

Here, the superscript $k$ in a map $f$ indicate the $k^{\rm th}$ suspension of $f$. Obviously, each $\p_1^k$, as well as $\p_1$, has degree $2$. Consider the corresponding exact sequence $$\to[S^3,S^2]\stackrel{\hat{\p}_1^2}{\longrightarrow}[S^3,S^2]\stackrel{\hat{\q}_2^1}{\longrightarrow}[\Sigma\RP^2,S^2]\stackrel{\hat{\ell}_1^1}{\longrightarrow}[S^2,S^2]\stackrel{\hat{\p}_1^1}{\longrightarrow}[S^2,S^2]\to$$

For $k=2,3$ we have $[S^k,S^2]=\pi_k(S^2)\approx\Z$ and the homomorphism $\hat{\p}_1^{k-1}$ corresponds to the multiplication by $2$. By exactness, $\hat{\ell}_1^1=0$ and $\hat{\q}_2^1$ is surjective. Hence $$[\Sigma\RP^2,S^2]={\im}(\hat{\q}_2^1)\approx\pi_3(S^2)/{\im}(\hat{p}_1^2)\approx\Z_2.$$

Thus, in the exact sequence (\ref{Equation-Exact-Sequence-From-Puppe}), $\hat{\p}_2^{1}$ corresponds to a homomorphism from $\Z_2$ into $\Z$, which forces it to be null.
\end{proof}

Next we prove Theorem \ref{Theorem-Homotopy-Class}.

\begin{proof}[Proof of Theorem \ref{Theorem-Homotopy-Class}]
As we have remarked after the statement of Theorem \ref{Theorem-Homotopy-Class}, it remains to prove the equality $[\RP^3,S^2]=\{[h'_n] : n\in\Z\}$.

By Lemmas \ref{Lemma-p2} and \ref{Lemma-p21}, the exact sequence (\ref{Equation-Exact-Sequence-From-Puppe}) gives rise to the short exact sequence of sets
\begin{equation*}\label{Equation-Short-Exact-Sequence}
0\to\pi_3(S^2)\stackrel{\hat{\q}_3}{\longrightarrow}[\RP^3,S^2]\stackrel{\hat{\ell}_2}{\longrightarrow}\{0,[\q_2]\}\to0
\end{equation*}

It follows that $\hat{\ell}_2$ is surjective and so \begin{equation}\label{Equation-Homotopy-Class}
    [\RP^3,S^2]=\hat{\ell}_2^{-1}(0)\sqcup\hat{\ell}_2^{-1}([\q_2])={\im}(\hat{\q}_3)\sqcup\hat{\ell}_2^{-1}([\q_2]).
\end{equation}

Since $\pi_3(S^2)\approx\Z$ is generated by $[h]$, we have \begin{equation}\label{Equation-im-q3}
    {\im }(\hat{\q}_3)=\big\{[nh\circ\q_3] : n\in\Z\big\}=\{[h'_{2n}] : n\in\Z\},
\end{equation}

Next we prove the equality \begin{equation}\label{Equation-pre-l2}
    \hat{\ell}_2^{-1}([\q_2])=\{[h'_{2n+1}] : n\in\Z\}.
\end{equation}

Firstly we observe that, from Proposition \ref{Proposition-Factorization-Hopf-Map} and Equations (\ref{Equation-Homotopy-Class}) and (\ref{Equation-im-q3}), we have the inclusion $\{[h'_{2n+1}] : n\in\Z\}\subseteq \hat{\ell}_2^{-1}([\q_2])$. In particular, $[h']\in\hat{\ell}_2^{-1}([\q_2])$.

Consider the natural action of the group $\pi_3(S^2)$ on $[\RP^3,S^2]$. By \cite[Proposition 2.48, p.\,30]{SW}, this action has the following property: for $[f],[g]\in[\RP^3,S^2]$ we have $\hat{\ell}_2([f])=\hat{\ell}_2([g])$ if and only if there exists $[\sigma]\in\pi_3(S^2)$ such that $[\sigma]\!\cdot\![f]=[g]$.

It follows that the orbit $\pi_3(S^2)[f]$ of each element $[f]\in\hat{\ell}_2^{-1}([\q_2])$ is whole the set $\hat{\ell}_2^{-1}([\q_2])$. In particular, we have $\hat{\ell}_2^{-1}([\q_2])=\pi_3(S^2)[h']$.

In order to find this orbit, we observe that, for each element $n[h]\in\pi_3(S^2)$, we have $n[h]\!\cdot\![h']=[\gamma_n]$, in which $\gamma_n$ is the composed map $$\gamma_n:\RP^3\stackrel{\omega}{\longrightarrow}S^3\vee\RP^3\stackrel{nh\vee h'}{\longrightarrow} S^2\vee S^2\stackrel{\nabla}{\longrightarrow} S^2,$$  where $\omega$ is the map which collapses the boundary of a $3$-disc embedded in the interior of the $3$-cell of $\RP^3$ and $\nabla$ is the folding map given by $\nabla(x,x_0)=x=\nabla(x_0,x)$.

It is not hard to see that $[\gamma_n\circ\p_3]=(2n+1)h$ and so $[\gamma_n]$ is the unique element in the orbit $\pi_3(S^2)[h']$ satisfying $\hat{\p}_3([\gamma_n])=(2n+1)h$. On the other hand, we have $[h'_{2n+1}]\in\hat{\ell}_2^{-1}([\q_2])=\pi_3(S^2)[h']$ and, by Proposition \ref{Proposition-Factorization-Hopf-Map}, $\hat{\p}_3([h'_{2n+1}])=(2n+1)h$. It follows that $[\gamma_n]=[h'_{2n+1}]$.

This proves that $\pi_3(S^2)[h']=\{[h'_{2n+1}] : n\in\Z\}$ and Equation (\ref{Equation-pre-l2}) follows. 

Finally, the equality $[\RP^3,S^2]=\{[h'_n] : n\in\Z\}$ follows from Equations (\ref{Equation-Homotopy-Class}), (\ref{Equation-im-q3}) and (\ref{Equation-pre-l2}).
\end{proof}

The following result is a useful consequence of Proposition \ref{Proposition-Factorization-Hopf-Map} and Theorem \ref{Theorem-Homotopy-Class}.

\begin{corollary}\label{Corollary-Fibration}
Let $f:\RP^3\to S^2$ be a map which is not null-homotopic. Then, there exists either a map $g_k:\RP^3\to\RP^3$ of degree $2k+1$ or a map $g_k:\RP^3\to S^3$ of degree $k$, such that $f=h'\circ g_k$ or $f=h\circ g_k$, respectively.
\end{corollary} 
\begin{proof} By Proposition \ref{Proposition-Factorization-Hopf-Map} and Theorem \ref{Theorem-Homotopy-Class}, there exist a map $g_k'$ as claimed in the Corollary where instead of the equality $f=h'\circ g_k'$ or $f=h\circ g_k'$ we have $f\simeq h'\circ g_k'$ or $f\simeq h\circ g_k'$, respectively. Using the lifting property of the fibrations $h':\mathbb{R}P^3\to S^2$ 
and $h: S^3\to S^2$, we can replace $g_k'$ by a map $g_k$ for which the equality holds.  
\end{proof} 


\section{Minimizing roots of maps into $S^2$}\label{Section-Roots}

\hspace{4mm} The Hopf fibration $h:S^3\to S^2$ has the property that $h^{-1}(y)$ is a circle for each $y\in S^2$. In particular, there exists a point $y_0\in S^2$ such that $h^{-1}(y_0)$ is the circle $$S_1=\{(z_1,0)\in\C^2 : |z_1|=1\}\subset S^3.$$

In fact, the orbit $S^1(1,0)$ of the point $(1,0)\in S^3$, under the action $S^1\times S^3\to S^3$ given by $(\lambda,(z_1,z_2))\mapsto(\lambda z_1,\lambda z_2)$, is the circle $S_1$, which implies that, for the point $y_0=h(1,0)\in S^2$, one has $h^{-1}(y_0)=S_1$.

\begin{theorem}\label{Theorem-Hopf-Minimum}
For each $n\neq0$, there exists a map $\varphi_n\in n[h]$ such that $\varphi_n^{-1}(y_0)=S_1$.
\end{theorem}
\begin{proof}
Consider the map $\A_n:S^3\to S^3$ defined by $\A_n(z_1,z_2)=(z_1^n,z_2)/|(z_1^n,z_2)|$, as defined in the proof of Proposition \ref{Proposition-Factorization-Hopf-Map}. Then $\deg(\A_n)=n$ and $\A_n^{-1}(S_1)=S_1$. Define $\varphi_n=h\circ\A_n$. Then $\varphi_n\in n[h]$ and $$\varphi_n^{-1}(y_0)=\A_n^{-1}(h^{-1}(y_0))=\A_n^{-1}(S_1)=S_1.$$

\vspace{-7mm}
\end{proof}

Since $\pi_3(S^2)=\{n[h] : n\in\Z\}$, Theorem \ref{Theorem-Hopf-Minimum} shows that each non null-homotopic map $f:S^3\to S^2$ is homotopic to a map $\varphi$ such that $\varphi^{-1}(y_0)=S_1$. On the other hand, a null-homotopic map $f_0:S^3\to S^2$ (i.e. $f_0\in0[h]$) is homotopic to a map $\varphi_0$ such that $\varphi_0^{-1}(y_0)=\varnothing$ (just take $\varphi_0$ to be the constant map at a point $y_1\neq y_0$).

\vspace{2mm}
The following result is the analogous of Theorem \ref{Theorem-Hopf-Minimum} for maps from $\RP^3$ into $S^2$.

\begin{theorem}\label{Theorem-HopfLine-Minimum}
For each $n\neq0$, there exists a map $\varphi_n'\in[h_n']$ such that $\varphi_n'^{-1}(y_0)$ is a circle.
\end{theorem}
\begin{proof}
We recall that the maps $h'_n:\RP^3\to S^2$ are defined in different ways according to the parity of $n$. So, we approach separately the cases $n$ even and $n$ odd.

We consider first the even case. For each $n\neq0$, the map $h_{2n}'$ is defined by the composition $h_{2n}'=nh\circ\q_3$. Then, by defining $\varphi_{2n}'=\varphi_n\circ\q_3$, we have  $\varphi_{2n}'\simeq h_{2n}'$ and  $$\varphi_{2n}'^{-1}(y_0)=\q_3^{-1}(\varphi_n^{-1}(y_0))=\q_3^{-1}(S_1).$$

Since $\q_3:\RP^3\to S^3$ is the map which collapses the $2$-skeleton $\RP^2\subset\RP^3$ to a point, which may be chosen to be a point not belonging to $S_1$, it follows that $\q_3^{-1}(S_1)$ is a circle.

For the odd case, we start by analysing the map $h'=h_1':\RP^3\to S^2$. Since $h'\circ\p_3=h$, we have $$S_1=h^{-1}(y_0)=\p_3^{-1}(h'^{-1}(y_0)) \quad{\rm and \ so}\quad h'^{-1}(y_0)=\p_3(S_1).$$

Obviously, $\p_3(S^1)$ is a circle in $\RP^3$ and so we have proved the result for $n=1$.

For an arbitrary odd integer $n$, the map $h_n':\RP^3\to S^2$ is defined by $h'_n=h'\circ\A'_n$, where, as in Section \ref{Section-Hopf-Map}, the map $\A'_n$ for $n$ 
odd is the self-map of $\RP^3$ induced on the quotient by the map $\A_n:S^3 \to S^3$, that is, $\A'_n\circ\p_3=\p_3\circ\A_n$. Hence,  $$h_n'^{-1}(y_0)=\A_{n}'^{-1}(h'^{-1})(y_0)=\A_n'^{-1}(\p_3(S_1))=\p_3(S_1).$$

\vspace{-9mm}
\end{proof}

Since $[\RP^3,S^2]=\{[h_n'] : n\in\Z\}$, Theorem \ref{Theorem-HopfLine-Minimum} implies that each non null-homotopic map $f':\RP^3\to S^2$ is homotopic to a 
map $\varphi'$ such that $\varphi'^{-1}(y_0)$ is a circle. On the other hand, a null-homotopic map $f_0':\RP^3\to S^2$ is homotopic to a map $\varphi_0'$ such that $\varphi_0'^{-1}(y_0)=\varnothing$.

\vspace{2mm}
In order to conclude this section, we show that for any non null-homotopic map $f$ from either $S^3$ or $\RP^3$ into $S^2$, we have ${\rm rank}(\check{H}^1(f^{-1}(y_0)))\geq1$, which means that, in a sense, the maps $\varphi_n$ and $\varphi_n'$ given in Theorems \ref{Theorem-Hopf-Minimum} and \ref{Theorem-HopfLine-Minimum}, respectively, realize the minimal root sets in its homotopy classes. We clarify this after the theorem.

\begin{theorem}\label{Theorem-H1neq0}
Let $f$ be a map from either $S^3$ or $\RP^3$ into $S^2$. If $f$ is not null-homotopic, then ${\rm rank}(\check{H}^1(f^{-1}(y_0)))\geq1$.
\end{theorem} 
\begin{proof} 
The proof for a map from $S^3$ into $S^2$ is similar and simpler than the proof for a map from $\RP^3$ into $S^2$. So, we present just the proof for this second setting.

Let $f:\RP^3\to S^2$ be a map which is not null-homotopic. From Corollary \ref{Corollary-Fibration} we have either $f=h\circ g$ or $f=h'\circ g$ for a map $g:\RP^3\to S^3$ or $g:\RP^3\to\RP^3$, respectively, of non zero degree. The proof is similar in both cases; we detail just the first of them and left the second for the reader.

Consider $f=h\circ g$, so that $f^{-1}(y_0)=g^{-1}(h^{-1}(y_0))=g^{-1}(S_1)$. We apply Proposition 10.2 of \cite[Chapter VIII, p.\,309]{Do} for the case in which $M'=\RP^3$, $N=S^3$, $K=S_1$, $L=\varnothing$ and the given map is $g$. For that, we must show that the hypotheses of the Proposition are satisfied. First we observe that the restriction of the Hopf fibration $h:S^3\to S^2$ over $S^2\backslash\{y_0\}$ is again a locally trivial fibration $S^1 \to S^3 \backslash S_1\to S^2\backslash\{y_0\}$ over an open $2$-disk, so the total space $S^3 \backslash S_1$ has the homotopy type of $S^1$ (in fact it is homeomorphic to the product of the open $2$-disk times $S^1$). Thus $H_i(S^3\backslash S_1)=0$ for $i=2,3$. Let us consider  the induced homomorphisms between the long exact sequences in homology  of the pairs $(\RP^3,\RP^3\backslash g^{-1}(S_1))$ and $(S^3,S^3\backslash S_1)$, namely,

\vspace{-4mm} {\small 
$$\xymatrix{H_3(\RP^3\backslash g^{-1}(S_1)) \ar[r] \ar[d]^-{g|_{\ast}} &  H_3(\RP^3) \ar[r] \ar[d]^-{g_{\ast}} & H_3(\RP^3,\RP^3\backslash g^{-1}(S_1)) \ar[r] \ar[d]^-{(g,g|)_{\ast}} &  H_2(\RP^3\backslash g^{-1}(S_1)) \ar[d]^-{g|_{\ast}} \\ 0=H_3(S^3\backslash S_1) \ar[r] & H_3(S^3) \ar[r]^-{\approx} & H_3(S^3,S^3\backslash S_1) \ar[r] & H_2(S^3\backslash S_1)=0. }$$}

\vspace{-2mm}\noindent Since $g$ is not null-homotopic, we have that $g_{*}$ is multiplication by $r\ne 0$. So the fundamental class $o_{S_1}\in H_3(S^3,S^3\backslash S_1)$ around the compact subset $S_1\subset S^3$ has the property that $r\!\cdot\!o_{S_1}$ belongs to the image 
of ${(g,g|)_{\ast}}$. This implies that the remains hypothesis of Proposition 10.2 of \cite[Chapter VIII, p.\,309]{Do} are satisfied. From that Proposition, there exists a sequence of homomorphisms $$\check{H}^1(S_1)\to\check{H}^1(g^{-1}(S_1))\to\check{H}^1_c(g^{-1}(S_1))\to\check{H}^1(S_1)$$ whose composite equal to $r$-times the identity homomorphism. Since $\check{H}^1(S_1)\approx\Z$, it follows that $\check{H}^1(g^{-1}(S_1))$ contains a infinite cyclic subgroup. The result follows.
\end{proof}

In view of Theorem \ref{Theorem-H1neq0}, we can state that, for a non null-homotopic map $f$ from either $S^3$ or $\RP^3$ into $S^2$, one can not expect to have the root set $f^{-1}(y_0)$ properly contained in a circle. Therefore, it is reasonable to state that, in this situation, the minimal set for the root problem must be a circle.

In fact, according Definition \ref{Definition-Minimal}, Theorems \ref{Theorem-Hopf-Minimum}, \ref{Theorem-HopfLine-Minimum} and \ref{Theorem-H1neq0} together imply that, in both cases, $X=S^3$ and $X=\RP^3$, the circle is the minimal root set for any non-trivial homotopy class $\alpha\in[X,S^2]$.


At this point we propose a comparison with the result presented in \cite[Example 1.3]{KP}, for $n=2$ and $n'=1$. It follows from that result that for a non null-homotopic map $f:S^3\to S^2$, the minimal root set is path connected and infinite. Our conclusion goes further, proving that the minimal root set is a circle.


\section{Minimizing roots of maps into $\RP^2$}\label{Section-Maps-Into-RP2}

\hspace{5mm} In this section, we solve the root problem for maps from $S^3$ and $\RP^3$ into $\RP^2$. Precisely, we prove a minimum theorem for roots. 

The first step is describing the sets $[S^3,\RP^2]$ and $[\RP^3,\RP^2]$. As before, $\p_2:S^2\to \RP^2$ is the double covering map.

\begin{lemma}\label{Lemma-Lifting}
The function $[f]\mapsto[\p_2\circ f]$ provides bijections from $[S^3, S^2]$ onto $[S^3,\RP^2]$ and from $[\RP^3,S^2]$ onto $[\RP^3,\RP^2]$.
\end{lemma}
\begin{proof} In both cases, the result follows from the {\it homotopy lifting property} for the covering map $\p_2:S^2\to\RP^2$. In the fist case, this is clear, since it is obvious that each map $f:S^3\to\RP^2$ lifts through $\p_2$. We claim that the same happens with each map $f:\RP^3\to\RP^2$, which is equivalent to say that each such a map induces the trivial homomorphism on fundamental groups. Suppose there exists a map $f:\RP^3\to\RP^2$ whose induced homomorphism $f_{\#}:\pi_1(\RP^3)\to\pi_1(\RP^2)$ is not trivial. Then $f_{\#}$ is an isomorphism and the ring homomorphism $f^{\ast}:H^{\ast}(\RP^2;\Z_2)\to H^{\ast}(\RP^3;\Z_2)$ maps the generator $\alpha\in H^1(\RP^2;\Z_2)$ to the generator $\beta\in H^1(\RP^3;\Z_2)$. However, the cohomology rings of $\RP^2$ and $\RP^3$ (with coefficients in $\Z_2$) are, respectively, $H^{\ast}(\RP^2;\Z_2)\approx\Z_2[\alpha]/(\alpha^3)$ and $H^{\ast}(\RP^3;\Z_2)\approx\Z_2[\beta]/(\beta^4)$; see \cite[Theorem 3.12, p.\,212]{Hatcher}. This provides the contradiction $0=f^{\ast}(\alpha^3)=\beta^3\neq0$.
\end{proof}

Next result is the minimum theorem for roots of maps from $S^3$ and $\RP^3$ into $\RP^2$.

In Section \ref{Section-Roots}, we fixed the point $y_0=h(1,0)\in S^2$ to be the base point. Here, we consider its projection $\bar{y}_0=\p_2(y_0)\in\RP^2$.

\begin{theorem}\label{Theorem-Min-Roots-RP2}
Let $f$ be a map from either $S^3$ or $\RP^3$ into $\RP^2$. If $f$ is not null-homotopic, then ${\rm rank}(\check{H}^1(f^{-1}(\bar{y}_0)))\geq 2$. Further, $f$ is homotopic to a map $\varphi$ such that $\varphi^{-1}(\bar{y}_0)$ is the disjoint union of two circles.  
\end{theorem} 
\begin{proof} Let $f: X\to\RP^2$ be a map which is not null-homotopic, where $X$ is either $S^3$ or $\RP^3$. By the proof of Lemma \ref{Lemma-Lifting}, $f$ lifts through $\p_2$ to a map $\tilde{f}:X\to S^2$ which, obviously, is not null-homotopic. Then $f^{-1}(\bar{y}_0)$ is the disjoint union $\tilde{f}^{-1}(y_0)\cup\tilde{f}^{-1}(-y_0)$. By Theorem \ref{Theorem-H1neq0}, we have both ${\rm rank}(\check{H}^1(\tilde{f}^{-1}(y_0)))\geq1$ and ${\rm rank}(\check{H}^1(\tilde{f}^{-1}(-y_0)))\geq1$. The first statement of the theorem follows.

In order to prove the second part, it suffices to find a map $g:X\to S^2$ homotopic to $\tilde{f}$ and two points $y_1,y_2\in S^2$ such that both $g^{-1}(y_1)$ and $g^{-1}(y_2)$ are circles.

Consider the points $y_1=h(1,0)$ and $y_2=h(0,1)$ in $S^2$, so that $h^{-1}(y_1)$ and $h^{-1}(y_2)$ are, respectively, the circles $$S_1=\{(z_1,0)\in\C^2 : |z_1|=1\} \quad{\rm and}\quad S_2=\{(0,z_2)\in\C^2 : |z_2|=1\}.$$

For the rest of the proof, we consider separately the cases $X=S^3$ and $X=\RP^3$.

For $X=S^3$, the map $\tilde{f}$ is homotopic to $nh$ for some $n\neq0$. Define the map $g=h\circ\A_n:S^3\to S^2$, as in the proof of Theorem \ref{Theorem-Hopf-Minimum}. Then $g$ is homotopic to $\tilde{f}$ and, for $i=1,2$, we have  $g^{-1}(y_i)=\A_n^{-1}(h^{-1}(y_i))=\A_n^{-1}(S_i)=S_i$.

For $X=\RP^3$, it follows from Proposition \ref{Proposition-Factorization-Hopf-Map} that $\tilde{f}$ is homotopic to $h'_n$ for some $n\neq0$. Define the map $g=\varphi_n':\RP^3\to S^2$ as in the proof of Theorem \ref{Theorem-HopfLine-Minimum}. Then $g$ is homotopic to $\tilde{f}$ and, since $y_1$ is is fact the point $y_0$ fixed in Section \ref{Section-Roots} to be the base point, it follows from Theorem \ref{Theorem-HopfLine-Minimum} that $g^{-1}(y_1)$ is a circle. That also $g^{-1}(y_2)$ is a circle follows {\it mutatis mutandis} the proof of Theorem \ref{Theorem-HopfLine-Minimum}.
\end{proof}

In view of Definition \ref{Definition-Minimal}, Theorem \ref{Theorem-Min-Roots-RP2} shows that the minimal root set for a non null-homotopic map $f$ from either $S^3$ or $\RP^3$ into $\RP^2$ is the disjoint union of two circles. This find fits and improves Theorem 1.17 of \cite{K3}, in the particular case in which the domain is $S^3$, $\mathbb{K}=\R$, $m=3$, $n'=2$ and one of the maps is the constant map at $\bar{y}_0$. In fact, it follows from that theorem that for a non null-homotopic map $f:S^3\to\RP^2$, one has ${\rm MCC}(f,\bar{y}_0)=2$ and ${\rm MC}(f,\bar{y}_0)=\infty$.


\section*{Acknowledgments}

\hspace{4mm} The second author is partially sponsored by Projeto Tem\'atico FAPESP: {\it Topologia Alg\'ebrica, Geom\'etrica e Diferencial} -- 2016/24707-4 (Brazil). 

The partnership of this work began during a visit of the second author to the first and third ones at {\it Universidade Federal de Uberl\^andia}.



\vspace{2mm}


\rule{4.5cm} {1pt} \vspace{2mm}

 {\sc Marcio Colombo Fenille} ({\sl mcfenille@gmail.com})
 
 Universidade Federal de Uberl\^andia -- Faculdade de Matem\'atica.
 
 Av.\,Jo\~ao Naves de \'Avila, 2121, Santa M\^onica, 38408-100, Uberl\^andia MG, Brasil.
 
 \vspace{4mm}
 
 {\sc Daciberg Lima Gon\c calves} ({\sl dlgoncal@ime.usp.br}) 
 
 Universidade de S\~ao Paulo -- Instituto de Matem\'atica e Estat\'istica.
 
  Rua do Mat\~ao, 1010, Cidade Universit\'aria, 05508-090, S\~ao Paulo SP, Brasil.
 
 \vspace{4mm}
 
 {\sc Gustavo de Lima Prado} ({\sl glprado@ufu.br})
  
 Universidade Federal de Uberl\^andia -- Faculdade de Matem\'atica.
 
 Av.\,Jo\~ao Naves de \'Avila, 2121, Santa M\^onica, 38408-100, Uberl\^andia MG, Brasil.

\vspace{3mm} $$\star \ \star \ \star$$

\end{document}